\documentclass[11pt]{article}
\usepackage{geometry,amsthm,amsmath,amssymb, graphicx, float, enumerate}
\usepackage{amsrefs}
\restylefloat{table}
\restylefloat{figure}

\newtheorem{thm}{Theorem}[section]

\newtheorem{cor}[thm]{Corollary}

\newtheorem{lemma}[thm]{Lemma}
\newtheorem{prop}[thm]{Proposition}

\theoremstyle{definition}
\newtheorem{defn}[thm]{Definition}
\theoremstyle{remark}
\newtheorem{rem}[thm]{Remark}

\newtheorem*{question*}{Question}
\newtheorem*{answer*}{Answer}

\def\eps{\epsilon}

\def\Snm1{\mathbb S^{n-1}}

\def\C{\mathbb C}
\def\R{\mathbb R}

\title{Algorithms and error bounds for noisy phase retrieval\\ with low-redundancy frames}
\date{}
\author{Bernhard G. Bodmann\thanks{Email: bgb@math.uh.edu, Tel.: 713 743 3581}\ \ and Nathaniel Hammen\\
\small 651 Philip G. Hoffman Hall, Department of Mathematics, University of Houston\\ \small Houston, TX 77204-3008}

\begin{document}
\maketitle

\begin{abstract}
The main objective of this paper is to find algorithms accompanied by explicit error bounds for phase retrieval
from noisy magnitudes of frame coefficients when the underlying frame has a low redundancy. 
We achieve these goals with
frames consisting of $N=6d-3$ vectors spanning a $d$-dimensional 
complex Hilbert space. The two algorithms we use, phase propagation or the kernel method, are polynomial time in the dimension $d$. 
To ensure a successful approximate recovery, we assume that the noise is sufficiently small compared to the squared norm of the vector
to be recovered. In this regime, the  error bound is inverse proportional to the signal-to-noise ratio. Upper and lower bounds on the sample
values of trigonometric polynomials are a central technique in our error estimates.
\end{abstract}

\section{Introduction}
When optical, electromagnetic or acoustic signals are measured,
often the measurement apparatus records an intensity, the magnitude of
the signal amplitude, while discarding phase information.
This is the case for X-ray crystallography \cite{Patterson,Millane_90}, many optical and acoustic systems \cite{Walther,Rabiner}, 
and also an intrinsic feature of quantum measurements \cite{Finkelstein_QuantumCom04,Gross_2014}. 
Phase retrieval is the procedure of determining missing phase information from suitably chosen intensity measurements, possibly 
with the use of additional signal characteristics \cite{Oppenheim_Phase_80,Fienup,Walther,Heinosaari_QuantumTom_13,MV13,Marchesini_AltProj14}. 
Many of these instances of phase retrieval are related to the Fourier transform \cite{Akutowicz1956,Akutowicz1957,Fienup_93}, but it is also of interest to study this
problem from an abstract point of view, using the magnitudes of any linear measurements
to recover the missing information. Next to infinite dimensional signal models \cite{PYB_JFAA14}, the finite dimensional case has received 
considerable attention in the past years \cite{Balan_RecWithoutPhase_06,Balan_Painless_09,Bandeira_4NConj,CandesEldar_PhaseRetrieval,Candes_PhaseLift, Demanet_PhaselessLinMeas13, CandesLi_FCM13,Waldspurger_PR14}. 
In this case, the signals are vectors in a finite dimensional Hilbert space $\mathcal H$ and
one chooses a frame $\{f_j\}_{j=1}^N$ to obtain for
each $x \in \mathcal H$ the magnitudes of the inner products with the frame vectors, $\{|\langle x, f_j \rangle |\}_{j=1}^N$.
When recovering signals, we allow for
a remaining undetermined global phase factor, meaning we identify vectors  in
the Hilbert space $\mathcal H$ that differ by a unimodular factor $\omega$, $|\omega|=1$,
in the real or complex complex case. 
Accordingly, we associate the equivalence class $[x]=\mathbb T x = \{\omega x: |\omega|=1 \}$
with a representative $x \in \mathcal H$ and consider the quotient space $ \mathcal H/\mathbb T$ as the domain
of the magnitude measurement map $\mathcal A:  \mathcal H/\mathbb T \to \mathbb R_+^N, \mathcal A([x])=(|\langle x, f_j\rangle|^2)_{j=1}^N$.
The map $\mathcal A$ is well defined, because $|\langle \omega x, f_j \rangle|^2$ does not depend
on the choice of $\omega \in \mathbb T$.  
The metric on $\mathcal H/\mathbb T$ relevant for the accuracy of signal recovery is the quotient metric $\rho$, which assigns to elements $[x]$ and $[y]$
with representatives $x, y \in \mathcal H$
the distance $\rho([x],[y])=\min_{|\omega|=1}\|x-\omega y\|$.

The case of signals in real Hilbert spaces is now fairly well understood
\cite{Balan_RecWithoutPhase_06,BCE07,Bandeira_4NConj}, while complex
signals still pose many open problems. 
When the number of measured magnitudes is allowed to grow at a sufficient rate,
then techniques from low-rank matrix completion are applicable to phase retrieval \cite{CandesEldar_PhaseRetrieval,Candes_PhaseLift, Demanet_PhaselessLinMeas13, CandesLi_FCM13,Waldspurger_PR14},
providing stable recovery from noisy measurements. Other methods also achieve stability  by a method that 
locally patches the phase information together \cite{Alexeev_PhaseRetrieval13,Yang_SampTA13,PYB_JFAA14}.
Recently, it was shown that for a vector in a complex $d$-dimensional 
Hilbert space, a generic choice of $4d-4$ linear measurements is sufficient to recover the vector up to a unimodular factor
from the magnitudes \cite{Conca_Algebraic13}, complementing an earlier result on a deterministic choice of $4d-4$ vectors \cite{BodmannHammen}.
Nevertheless, fully quantitative stability estimates were missing in this case of lowest redundancy  known   to be sufficient for recovery. 



A main objective of this paper is to find frames $\{f_j\}_{j=1}^N$ for the $d$-dimensional 
complex Hilbert space  such that $N$ is small and
the magnitude measurement map 
is injective on  $ \mathcal H/\mathbb T$ with explicit error bounds for the approximate recovery when the magnitude measurements
are affected by noise. More precisely, we find a left inverse of $\mathcal A$ which extends to
a neighborhood of the range of $\mathcal A$ and is Lipschitz continuous for all input signals whose signal-to-noise ratio is sufficiently large. 
We show that the recovery is implemented with an explicit algorithm
that restores the signal from measurements to a given accuracy in a number of operations that is polynomial in the dimension of the Hilbert space.
The algorithm 
can be chosen to be either phase propagation or what we call the kernel method, a special case of semidefinite programming. 
The smallest number of frame vectors
for which we could provide an algorithm with explicit error bounds is $N=6d-3$, as presented here.

To formulate the main result, it is convenient to take the Hilbert space as a space of polynomials $\mathcal P_d$ of maximal degree $d-1$ equipped with
the standard inner product, see Section~\ref{sec:poly}
for details. With this choice of Hilbert space, the magnitude measurements we use are expressed in terms of point evaluations. We 
let $\omega=e^\frac{2i\pi}{2d-1}$ denote the primitive $(2d-1)$-st root of unity and $\nu=e^\frac{2i\pi}d$ the primitive $d$-th root of unity.
For a polynomial $p \in \mathcal P_d$, the noiseless magnitude measurements are
$$
   (\mathcal A(p))_j = \left\{
\begin{array}{cl}
|p(\omega^j)|^2&\text{if }1\le j\le 2d-1 ,\\
|p(\omega^j)-p(\omega^j\nu)|^2&\text{if }2d\le j\le4d-2 ,\\
|p(\omega^j)-ip(\omega^j\nu)|^2&\text{if }4d-1\le j\le6d-3 \, .
\end{array}
\right.
$$

The noisy magnitude measurements $\tilde{\mathcal A}$ are obtained from perturbing the noiseless magnitudes with a vector $\epsilon \in \mathbb R^{6d-3}$,
$(\tilde{\mathcal A}(p, \epsilon))_j = (A(p))_j+\eps_j$, $j \in \{1, 2, \dots, 6d-3\}$.

Our main theorem states that for all measurement errors with a sufficiently small maximum noise component $\|\eps\|_\infty=\max_j |\epsilon_j|$, the noisy magnitude measurements determine an approximate
reconstruction of $p$ with an accuracy $O(\|\eps\|_\infty)$.
To state the theorem precisely involves several auxiliary quantities that all depend solely on the dimension $d$.
We let $r=\sin(\frac{2\pi}{(d-1)d^2})$ and choose a slack variable $\alpha \in (0,1)$
as well as $\beta=\frac{r^\frac{(d-1)d}2\left(\frac{d-1}{2d}\right)^d\frac2{d-1}}{\prod_{k=1}^{d-1}(r^k+1)}$.

\begin{thm}
Let $r, \alpha$ and $\beta$ be as above. For any nonzero analytic polynomial $p\in\mathcal P_d$ and $\eps\in\R^{6d-3}$ with $\|\eps\|_\infty\le\frac{\alpha(\beta\|p\|)^2}{2d-1}$, an approximation $\tilde p\in \mathcal P_d$ can be constructed from the perturbed magnitude measurements 
$\widetilde{\mathcal A}(p,\eps)$, such that 
if $\tilde C=\frac{(1+\sqrt2)(2d-1)\|\eps\|_\infty+d \|p\|^2}{(\beta\|p\|)^2(1-\alpha)}$ then the recovery error is bounded by
$$
\rho([p],[\tilde p])\le
\left(\frac{2+\sqrt2}{\beta^2(1-\alpha)}\frac{d-d\tilde C-1+\tilde C^d}{1-\tilde C}\sqrt d+\frac{1-\tilde C^d}{2\beta\sqrt{\frac1{\sqrt d}(1-\alpha)}}\right)\frac{d(2d-1)}{(1-\tilde C)}\frac{\|\eps\|_\infty}{\|p\|_2} \, .
$$
\end{thm}

The proof and the construction of approximate recovery proceeds in several steps: 
\begin{description}
\item[Step 1.] First, we augment the finite number of magnitude measurements to an infinite family of such measurements.
To this end, the Dirichlet Kernel is used to interpolate the perturbed measurements to functions on the entire unit circle. 
In the noiseless case, the magnitude measurements $\mathcal A(p)$ determine the values
$|p(z)|^2$, $|p(z)-p(z\nu)|^2$, and $|p(z)-ip(z\nu)|^2$ for each $z, |z| = 1$, because these are trigonometric polynomials
of degree at most $d-1$.
In the noisy case,  the interpolation using values from $\widetilde{\mathcal A}(p,\eps)$, yields trigonometric polynomials that differ 
from the unperturbed ones by at most $(2d-1)\|\eps\|_\infty$, uniformly on the unit circle. 
\item[Step 2.] We select a suitable set of non-zero magnitude measurements from the infinite family.
A lemma will show that there exists a $z_0$ on the unit circle such that the distance between any element of $\{\nu^jz_0\}_{j=1}^d$ and any roots of any non-zero truncation of 
the polynomial $p$ is at least $r$. The reason why we need to consider all non-zero truncations of the polynomial is that the influence
of the noise prevents us from determining the true degree of $p$. However, when the coefficients of leading powers are sufficiently
small compared to the noise, we can replace $p$ with a truncated polynomial without losing the order of approximation accuracy.
As a consequence, we show that for this $z_0$, $\min_{1 \le j \le d} |p(\nu^jz_0)|\ge m$ with some $m>0$ that only depends
on the dimension $d$ and the norm of $p$. 
Thus, if the noise is sufficiently small compared to the norm of the vector, then there is a similar lower bound on the real trigonometric polynomials that interpolate the noisy magnitude measurements. 
\item[Step 3.] 
In the last step, the reconstruction evaluates the trigonometric approximations at the 
sample points $\{\nu^jz_0\}_{j=1}^d$ and recovers an approximation to the equivalence class $[p]$.
It is essential for this step that the sample values are bounded away from zero
in order to achieve a unique reconstruction.
There are two algorithms considered for this, phase propagation, which
recovers the phase iteratively using the phase relation between sample points,
and the kernel method, which computes a vector in the kernel of
a matrix determined by the magnitude measurements. The error bound
is first derived for phase propagation and then related to that of the kernel method.
Both algorithms are known to be polynomial time, either from the explicit description,
or from results in numerical analysis \cite{GuEisenstat}.
\end{description}

The nature of the main result has also been observed in simulations; assuming an a priori bound on the magnitude of the noise 
results in a worst-case recovery error that grows at most inverse proportional to the
signal-to-noise ratio. Outside of this regime, the error is not controlled in a linear fashion.
To illustrate this, we include two plots of the typical behavior for the recovery error
for $d=7$. The range of the plots is chosen to show the behavior of the worst-case error in the
linear regime and also for errors where this linear behavior breaks down. 

We tested the algorithm on more than 4.5 million randomly generated polynomials with norm 1. 
When errors were graphed for a fixed polynomial, the linear bound for the worst-case error was confirmed, although the observed errors were many orders of magnitude less than the error bound given in this paper. A small number of polynomials we found exhibited a max-min value that is an order of magnitude smaller than that of all the other randomly generated polynomials. We chose the polynomial with the worst max-min value out of the 4.5 million that had been tried, and applied a random walk to its  coefficients, with steps of decreasing size that were accepted only if the max-min value decreased. The random walk terminated at a polynomial which provided an error bound that is an order of magnitude worse than any other polynomials that had been tested before. This numerically found, local worst-case polynomial is given by $p(z)=0.3114912-0.0519351i 
-(0.0367368+0.6727228i) z
-(0.2214904+0.1978638i) z^2
-(0.3210523+0.3897147i) z^3
-(0.1358901-0.1047726i) z^4
-(0.07403360+0.2281884i) z^5
-(0.12017811-0.04210790i) z^6$. The accuracy of the coefficients
displayed here is sufficient to reproduce the results initially obtained
with floating point coefficients of double precision. 
The errors resulting for this polynomial
in the linear and transition regimes are shown in Figures~\ref{fig1} and \ref{fig2}.
%
\begin{figure}[h]
\begin{center}
 \includegraphics[width=3.6in]{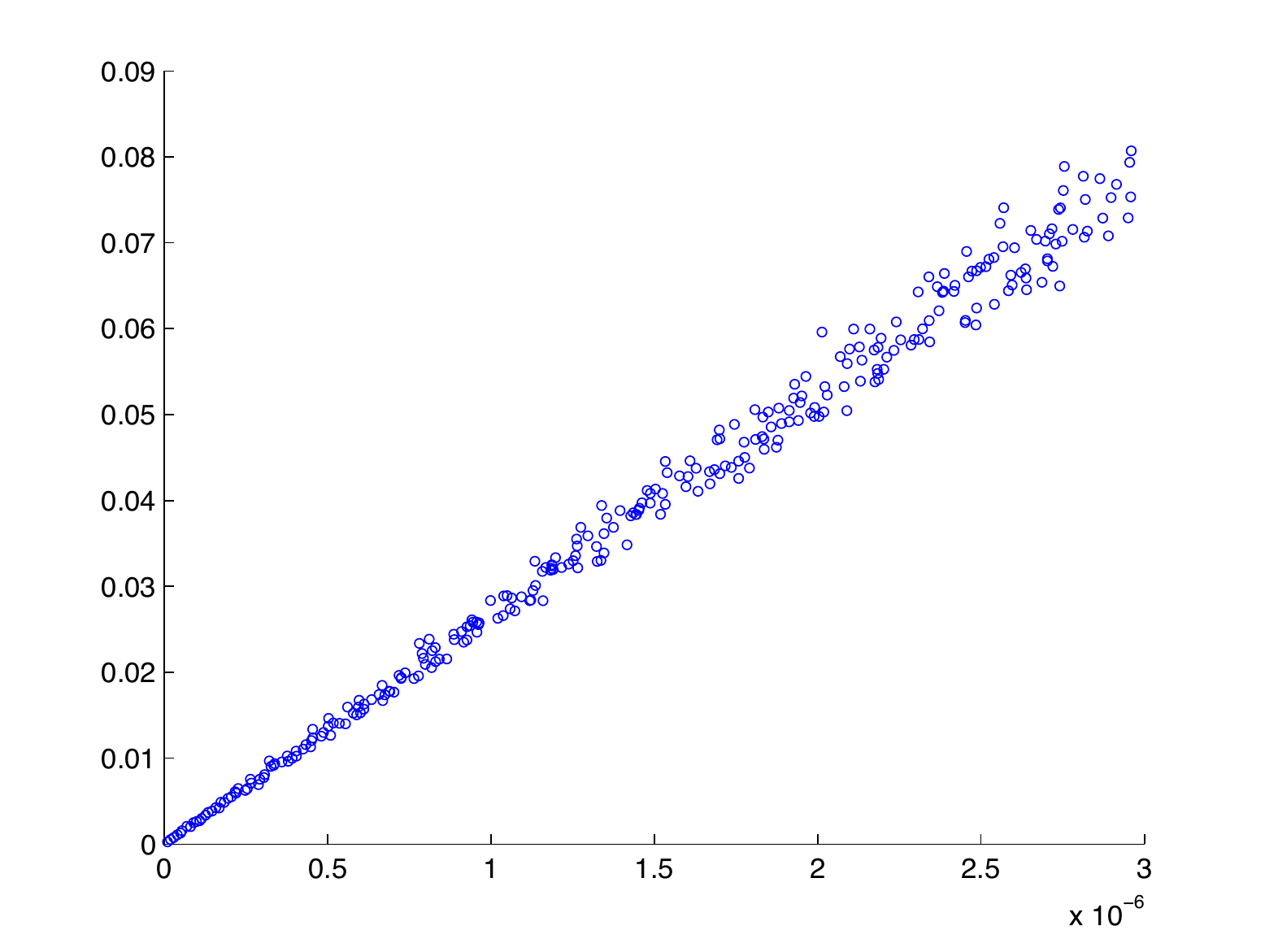} 
 \caption{Recovery error for the experimentally found worst-case normalized polynomial in dimension $d=7$ as a function of the maximal noise magnitude
 in the linear regime.}\label{fig1}
 \end{center}
 \end{figure}
 %
\begin{figure}[h]
\begin{center}
 \includegraphics[width=3.6in]{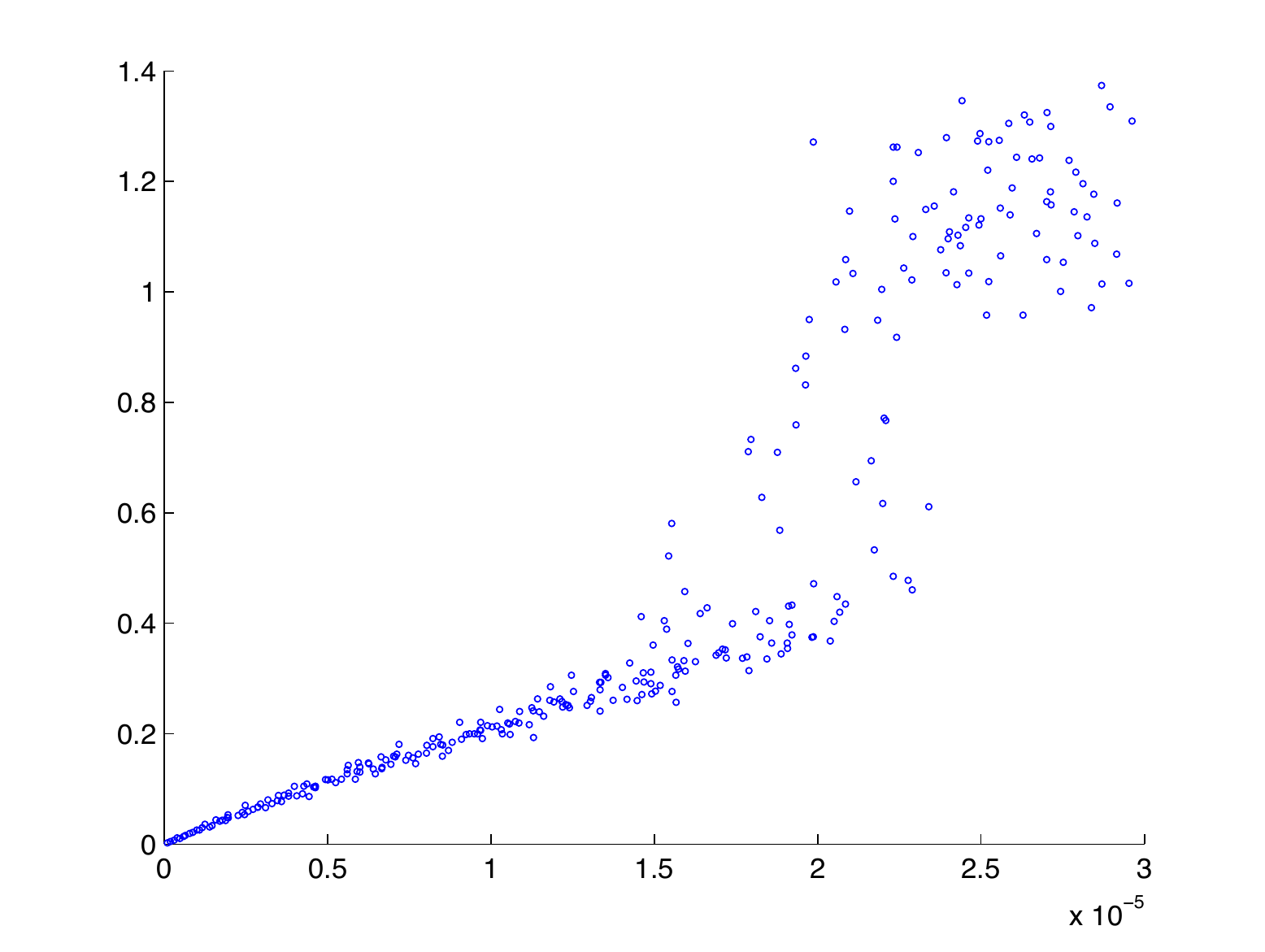}
 \caption{Recovery error for the experimentally found worst-case polynomial in dimension $d=7$ as a function of the maximal noise magnitude
 beyond the linear regime.}\label{fig2}
 \end{center}
 \end{figure}

\section{Noiseless recovery}

It is instructive to follow the construction of the magnitude measurements and the recovery strategy in the absence of noise.

\subsection{Recovery algorithms for  full vectors}\label{sec:full}

To motivate and prepare the recovery strategy, we compare two recovery methods, phase propagation and the kernel method, 
a simple form of semidefinite programming,
in the absence of noise and under additional 
non-orthogonality conditions
on the input vector. 

\begin{defn}
If $\{e_j\}_{j=1}^d$ is a basis for $\C^d$, and $x\in\C^d$ such that for all $j$ from $1$ to $d$, $\langle x,e_j\rangle\ne0$ then we call $x$ {\em full} with
respect to $\{e_j\}_{j=1}^d$. 
\end{defn}

We recall a well known result concerning recovery of full vectors  \cite{Finkelstein_QuantumCom04,Flammia_PureStates05}.
Let $x \in \mathcal H$ be full with respect to an orthonormal basis $\{e_j\}_{j=1}^d$.
For any $j$ from $1$ to $3d-2$, we define the measurement vector $f_j\in\C^d$ as
$$
f_j=\left\{
\begin{array}{cl}
e_j&\text{if }1\le j\le d\\
e_{j-d}-e_{j-d+1}&\text{if }d+1\le j\le2d-1\\
e_{j-(2d-1)}-ie_{j-(2d-1)+1}&\text{if }2d\le j\le3d-2
\end{array}
\right.
$$
The set $\{f_j\}_{j=1}^{3d-2}$ of measurement vectors is a frame for $\C^d$ because it contains a basis. Define the magnitude 
measurement map $\mathcal A_{\{e_j\}}:\C^d\to\R_+^{3d-2}$ by $\mathcal A_{\{e_j\}}(x)=(|\langle x,f_j\rangle|^2)_{j=1}^{3d-2}$. 

Recovery of full vectors with $3d-2$ measurements has been shown in \cite{Finkelstein_QuantumCom04}, and was proven to be minimal in \cite{Flammia_PureStates05}. We show recovery of full vectors
with $3d-2$ measurements using the measurement map $\mathcal A_{\{e_j\}}$ with two different recovery methods. The first is called phase propagation, the second
is a special case of semidefinite programming.

\subsubsection{Phase propagation}
The phase propagation method sequentially recovers the components of the vector, similar to the approach outlined in \cite{Balan_Painless_09},
see also \cite{Yang_SampTA13,Pohl_ICASSP14,PYB_JFAA14}.

\begin{prop}
For any vector $x\in\C^d$, if $\{e_j\}_{j=1}^d$ is an orthonormal basis with respect to which $x$ is full, then the vector $y=\frac{\langle e_1, x\rangle}{|\langle e_1, x\rangle|}x$ may be obtained by induction on the components of $y$, using the values of $\mathcal A_{\{e_j\}}(x)$.
\end{prop}
\begin{proof}
Without loss of generality, we assume that $\{e_j\}_{j=1}^d$ is the standard basis,
drop the subscript from $\mathcal A_{\{e_j\}}$  and abbreviate the components of the vector $x$ by
$x_j = \langle x, e_j\rangle$ for each
$j \in \{1, 2, \dots, d\}$, and similarly for $y$.
To initialize, we let $y_1=\sqrt{|\langle x,e_1 \rangle|^2}=|\langle x,e_1 \rangle|=|x_1|$ so that $y_1=\frac{\overline{x_1}}{|x_1|}x_1$.

For the $k$th  inductive step with $k<d$, we assume that we have constructed $y_k$ with the given information. We then let
$$
y_{k+1}
=\frac{1}{2\mathcal A(x)_k} \left( (1-i)\mathcal A(x)_k+(1-i)\mathcal A(x)_{k+1}
-\mathcal A(x)_{k+d}+i \mathcal A(x)_{k+2d+1} \right)y_k \, .
$$
Inserting the values for the magnitude measurements and by a fact similar to the polarization identity,
\begin{align*}
y_{k+1}
=&
\frac{1}{2|x_k|^2} \left((1-i)|x_k |^2+(1-i)|x_{k+1}|^2-|x_k-x_{k+1}|^2+i| x_k-ix_{k+1} |^2 \right)y_k\\
=&\frac{\overline{x_k}x_{k+1}}{|x_k |^2}y_k
=\frac{y_kx_{k+1}}{x_k}
=\frac{\frac{\overline{x_1}}{|x_1|}x_kx_{k+1}}{x_k}
=\frac{\overline{x_1}}{|x_1|}x_{k+1} \, .
\end{align*}
Iterating this,  we obtain $y=\frac{\overline{x_1}}{|x_1|}x$.
\end{proof}

\subsubsection{Kernel method}
Recovery by the kernel method minimizes the values of a quadratic form subject to a norm constraint, 
or equivalently, computes an extremal eigenvector for an operator associated with the quadratic form.
The operator we use is $Q_x=T^*_xT_x$, with
$$
   T_x = \sum_{j=1}^{d-1} (  |\langle x, e_j\rangle|^2 e_j \otimes e_{j+1}^* 
   - \langle e_j, x\rangle \langle x, e_{j+1} \rangle  e_j \otimes e_{j}^* )
$$
where each $e_j^*$ denotes the linear functional which is associated with
the basis vector $e_j$. The operator $T_x$ 
is indeed determined by the magnitude measurements. In particular,
the second term in the series is computed
 via the polarization-like identity
as in the proof of the preceding theorem,
$$
2 \langle e_j,x  \rangle\langle x,e_{j+1} \rangle= (1-i)\mathcal A_{\{e_j\}}(x)_k+(1-i)\mathcal A_{\{e_j\}}(x)_{k+1}
-\mathcal A_{\{e_j\}}(x)_{k+d}+i \mathcal A_{\{e_j\}}(x)_{k+2d+1} 
$$
for any integer $j$ from $1$ to $d-1$.

By construction, the rank of $Q_x$ is at most equal to $d-1$, because the range of $T_x$ is in the
span of $\{e_j\}_{j=1}^{d-1}$. In the  next theorem, we show that indeed the kernel of $T_x$, or equivalently, the kernel of $Q_x$,
is one dimensional, consisting of all multiples of $x$.

\begin{prop}
For any vector $x\in\C^d$, if $\{e_j\}_{j=1}^d$ is a basis with respect to which $x$ is full, 
then the null space of the operator $T_x$ is given by all complex multiples of $x$.
\end{prop}
\begin{proof}
As in the preceding proof, we let $\{e_j\}_{j=1}^d$ denote the standard basis.
Thus, using the measurements provided, we may obtain the quantity $\overline{x_j}x_{j+1}=\overline{\langle x,e_j \rangle}\langle x,e_{j+1} \rangle$. With respect to the basis $\{e_j\}_{j=1}^{d}$, let 
$S:\C^d\to\C^{d}$ be the left shift operator $S\left((y_j)_{j=1}^d\right)=(y_{j+1})_{j=1}^{d}$ where we extend the vector $y$ with the convention $y_{d+1}=0$.
We also define the multiplication operator $M_{\overline{x}Sx}:\C^{d}\to\C^{d}$ by the map $M_{\overline{x}Sx}\left((y_j)_{j=1}^{d}\right)=(\overline{x_j}x_{j+1}y_j)_{j=1}^{d}$, where again by convention we let $x_{d+1} = 0$. Similarly, we define the multiplication operator $M_{|x|^2}:\C^{d}\to\C^{d}$ by the map $M_{|x|^2}\left((y_j)_{j=1}^{d}\right)=(|x_j|^2y_j)_{j=1}^{d}$. Note that $M_{|x|^2}$ is invertible if and only if $|\langle x,e_j \rangle|\ne0$ for all $j$ from $1$ to $d$, which is true by assumption.  In terms of these operators, the operator $T_x:\C^d\to\C^{d}$ is expressed as $T_x=M_{|x|^2}S-M_{\overline{x}Sx}$. Then for any $c\in\C$
and $x \in \mathbb C^d$,
\begin{align*}
T_x(cx)=&M_{|x|^2}S(cx)-M_{\overline{x}Sx}(cx)\\
=&M_{|x|^2}\left((cx_{j+1})_{j=1}^{d}\right)-M_{\overline{x}Sx}\left((cx_j)_{j=1}^{d}\right)\\
=&(|x_j|^2cx_{j+1})_{j=1}^{d}-(\overline{x_j}x_{j+1}cx_j)_{j=1}^{d}\\
=&0
\end{align*}
so any complex multiple of $x$ is in the null space of this operator.

Conversely, assume that $y$ is in the null space of $T_x$. We use an inductive argument to show that for any $j$ from $1$ to $d$, $\frac{y_j}{x_j}=\frac{y_1}{x_1}$. The base case $j=1$ is trivial. For the inductive step, note that for any $j$ from $1$ to $d-1$,
$$
y_{j+1}-\frac{\overline{x_j}x_{j+1}}{|x_j|^2}y_j=M_{|x|^2}^{-1}\left(M_{|x|^2}(Sy-M_{\overline{x}Sx}) y\right)_j=(M_{|x|^2}^{-1}T_xy)_j=0 \, .
$$
Thus, $y_{j+1}=\frac{\overline{x_j}x_{j+1}}{|x_j|^2}y_j$, and because $x_j\ne0$ and $x_{j+1}\ne0$, we obtain
$$
\frac{y_{j+1}}{x_{j+1}}=\frac1{x_{j+1}}\frac{\overline{x_j}x_{j+1}}{|x_j|^2}y_j=\frac{y_j}{x_j} \, .
$$
We conclude that for any $j$ from $1$ to $d$, $y_j=\frac{y_1}{x_1}x_j$, so the vector $y$ is a complex multiple of $x$.
\end{proof}

Since the frame vectors used for the magnitude measurements contain an orthonormal basis, $\mathcal A_{\{e_j\}}$
determines the norm of $x$. This is sufficient to recover $[x]$.

\begin{cor}
If the vector $x \in \mathcal H$ is full with respect to the orthonormal basis $\{e_j\}_{j=1}^d$, then
the equivalence class $[x]$ is the solution of the problem
$$
     \arg\min \{  \|T_x y \|^2 : y \in \mathcal H, \|y\|^2=\sum_{i=1}^d \mathcal A(x)_i \} \, .  
$$
\end{cor}

Because the solution to the phase retrieval problem is obtained from the kernel of the linear operator $T_x$, or equivalently of $Q_x$,  
we may use methods from numerical linear algebra such as a rank-revealing QR factorization \cite{GuEisenstat} to recover the equivalence class $[x]$. 

\subsection{Augmentation and subselection of magnitude measurements} \label{sec:poly}

One of the main tools for the recovery procedure is that an entire family of magnitude measurements is determined
from the initial choice. We call this an augmentation of the measured values. From this family a suitable subset is chosen
which corresponds to a measurement of the form $\mathcal A_{\{e_j\}}$ related to an orthonormal basis as explained in the previous section.  

To describe the augmentation procedure we represent the $d$-dimensional vector to be recovered as an element of $\mathcal P_d$, the space of complex analytic polynomials of degree at most $d-1$ on the unit circle. This space is used to represent the vector because $\mathcal P_d$ is a reproducing kernel Hilbert space, and the magnitude squared of any element of $\mathcal P_d$ is an element of $\mathcal T_d$, the space of trigonometric polynomials of degree at most $d-1$ on the unit circle, which is itself a reproducing kernel Hilbert space. The space $\mathcal P_d$ is equipped with the scaled $L^2$ inner product on the unit circle such that for any $p,q\in\mathcal P_d$,
$$
\langle p,q\rangle
=\frac1{2\pi}\int_{[0,2\pi]}p(e^{it})\overline{q(e^{it})}dt \, .
$$
For any $p:z\mapsto\sum_{j=0}^{d-1}c_jz^j$ in $\mathcal P_d$, let $c$ be the vector of coefficients $(c_j)_{j=0}^{d-1}$ of $p$. Then by orthogonality, the norm induced by the inner product satisfies
$$
\|p\|^2
=\frac1{2\pi}\int_{[0,2\pi]}p(e^{it})\overline{p(e^{it})}dt
=\frac1{2\pi}\int_{[0,2\pi]}\sum_{j=0}^{d-1}|c_j|^2dt
=\|c\|_2^2 \, .
$$
If $K_w\in\mathcal P_d$ is defined such that $K_w(z)=\sum_{j=0}^{d-1}\overline{w^j}z^j$, then for any $p\in\mathcal P_d$ and any $z_0$ on the unit circle,
$$
p(z_0)=\sum_{j=0}^{d-1}c_jz_0^j=\langle p,K_{z_0} \rangle \, .
$$
Thus, these polynomials $K_w$ correspond to point evaluations, and linear combinations of these polynomials may be used as measurement vectors for the recovery procedure. If $\omega=e^\frac{2i\pi}{2d-1}$ is the $2d-1$-st root of unity and $\nu=e^\frac{2i\pi}d$ is the $d$-th root of unity, then for any $j$ from $1$ to $6d-3$, we define the measurement vector $\eta_j\in\mathcal P_d$ as
$$
\eta_j=\left\{
\begin{array}{cl}
K_{\omega^j}&\text{if }1\le j\le2d-1\\
K_{\omega^j}-K_{\omega^j\nu}&\text{if }2d\le j\le4d-2\\
K_{\omega^j}-iK_{\omega^j\nu}&\text{if }4d-1\le j\le6d-3
\end{array}
\right.
$$
Then the magnitude measurement map $\mathcal A:\mathcal P_d\to\R_+^{6d-3}$ defined in the Introduction satisfies $\mathcal A(p)=(|\langle p,\eta_j\rangle|^2)_{j=1}^{6d-3}$. 

The measurements are grouped into three subsets, corresponding to magnitudes of point evaluations,
magnitudes of differences, and magnitudes of differences between complex multiples of point values.
Each of these subsets can be interpolated to a family of measurements from which suitable representatives are
chosen. In order to simplify the recovery, we recall that for any $z$ with $|z|=1$ and $j \in \{1,2 , \dots , d-1\}$, $K_z$ and $K_{z \nu^j}$ are orthogonal 
because the series given by the inner product 
$\langle K_z, K_{z \nu^j}\rangle$ simply sums all the $d$-th roots of unity.


\begin{thm}
For any polynomial $p\in\mathcal P_d$, the measurements  $\mathcal A(p)$ determine 
the values of $\mathcal A_{\{\frac1{\sqrt d} K_{z_0\nu^j}\}}$ with
$z_0 \in \mathbb T$ such that
$ \{\frac1{\sqrt d} K_{z_0\nu^j}\}_{j=0}^{d-1} $ is an orthonormal basis with respect to which
$p$ is full.
\end{thm}
\begin{proof}
Let $D_{d-1}\in\mathcal T_d$ be the normalized Dirichlet kernel of degree $d-1$, so that for any $z$ in the unit circle $D_{d-1}(z)=\frac1{2d-1}\sum_{k=-(d-1)}^{d-1}z^k$. Then the set of functions $\{z\mapsto D_{d-1}(z\omega^{-l})\}_{l=1}^{2d-1}$ is orthonormal with respect to the $L^2$ inner product on the unit circle, and any $g\in\mathcal T_d$ can be interpolated as $g(z)=\sum_{l=1}^{2d-1}g(\omega^l)D_{d-1}(z\omega^{-l})$. Note that if $h\in\mathcal P_d$, then $\overline hh\in\mathcal T_d$. Thus, each of the functions $|p(z)|^2$, $|p(z)-p(z\nu)|^2$, and $|p(z)-ip(z\nu)|^2$, are in $\mathcal T_d$, and using the Dirichlet Kernel these functions may be interpolated from the values of $\mathcal A(p)$. So the values of each of these functions are known at all points on the unit circle, not just the points that were measured explicitly.

Let $\{z_1,\dots,z_m\}$ be the zeros of the polynomial $p$. Then the set $\{z_j\nu^k\}_{j=1,k=1}^{j=m,k=d}$ has finitely many elements, so we may choose a point $z_0$ on the unit circle such that $z_0\not\in\{z_j\nu^k\}_{j=1,k=1}^{j=m,k=d}$. Then for all $k$ from $1$ to $d$, $p(z_0\nu^k)\ne0$. Additionally, the set 
$\{z\mapsto\frac1{\sqrt d}K_{z_0\nu^j}(z)\}_{j=1}^d$ is an orthonormal basis for $\mathcal P_d$, so $p$ is a full vector with respect to this basis and the values of $\sqrt d|p(z)|^2$, $\sqrt d|p(z)-p(z\nu)|^2$, and $\sqrt d|p(z)-ip(z\nu)|^2$ at the points $\{z_0\nu^k\}_{k=1}^d$ correspond to the measurements $\mathcal A_{\{\frac1{\sqrt d}K_{z_0\nu^j}\}}(p)$. 
\end{proof}

Because the recovery procedure for full vectors has been established in Section~\ref{sec:full}, the phase-retrieval problem may now be solved using either procedure outlined there to obtain the equivalence class $[p]$.

\begin{cor}
For any polynomial $p\in\mathcal P_d$, the measurements  $\mathcal A(p)$ determine
$[p]$.
\end{cor}

Moreover, the number of points that need to be tested in order to find $z_0$ such that $p(z_0\nu^k)\ne 0$ for all $k \in \{1,2, \dots, d\}$
is at most $(d-1)d+1$, quadratic in $d$. This means, in the noiseless case either of the two equivalent methods for recovery requires a number of steps that
is polynomial in the dimension~$d$. 

\section{Recovery from noisy measurements}

For the purpose of recovery from noisy measurements, we consider the perturbed magnitude measurement map $\widetilde{\mathcal A}:\mathcal P_d\times\R^{6d-3}\to\R^{6d-3}$ by $\widetilde{\mathcal A}(p,\eps)=(|\langle p,\eta_j\rangle|^2+\eps_j)_{j=1}^{6d-3}$.

\subsection{A max-min principle for magnitude samples}
\label{sec:max-min}
In the presence of noise, choosing a basis such that $x$ is full with respect to that basis is not sufficient to establish a bound on the error of the recovered polynomial. A lower bound on the magnitude of the entries of $x$ in a such a basis is needed. Once such a lower bound is obtained, recovery may proceed as in the noiseless case, by  reduction of the phase retrieval problem to full vector recovery, with quantitative bounds on stability added for each step.
Obtaining the needed lower bound on the magnitude  requires a few lemmas and definitions.
\begin{lemma}
\label{lem:geom-trick}
If $a\in\C^d$ is represented as $a=(a_k)_{k=0}^{d-1}$, then for any $t\in(0,1)$, there exists at least one $n$ between $0$ and $d-1$ such that $|a_n|\ge\|a\|_1t^{d-n}(t^{-1}-1)$.
\end{lemma}
\begin{proof}
By way of contradiction, let $|a_j|<\|a\|_1t^{d-j}(t^{-1}-1)$ for all $j$ from $0$ to $d-1$. Then
$$
\|a\|_1
=\sum_{j=0}^{d-1}|a_j|
<\sum_{j=0}^{d-1}\|a\|_1t^{d-j}(t^{-1}-1)
=\|a\|_1(t^{-1}-1)\sum_{l=1}^dt^l
\le\|a\|_1
$$
This is a contradiction, so the claim holds.
\end{proof}

\begin{defn}
If $p\in\mathcal P_d$ such that $p(z)=\sum_{j=0}^{d-1}c_jz^j$ then we define 
for $n \in \{1,2, \dots, d\}$ 
the $n$-th truncation of $p$ to be the polynomial $p_n\in\mathcal P_n$ such that $p_n(z)=\sum_{j=0}^{n-1}c_jz^j$. 
\end{defn}

According to this definition, the $d$-th truncation of a polynomial is the polynomial itself. We also note that even if a polynomial is nonzero, it may 
have truncations that are zero polynomials.

To obtain a lower bound on the entries of a full vector originating from a polynomial, a lower bound on the distance between any basis element and any roots of any nonzero truncations of the base polynomial is needed.

\begin{lemma}
\label{lem:dist-bnd}
For any polynomial $p\in\mathcal P_d$, there exists a $z_0$ on the unit circle such that the linear distance between any element of $\{\nu^jz_0\}_{j=1}^d$ and any roots of any nonzero truncations of $p$ is at least $\sin(\frac{2\pi}{(d-1)d^2})$.
\end{lemma}
\begin{proof}
For any $n$ from $1$ to $d$, let $N_n$ be the number of distinct roots of the $n$-th truncation of $p$ if that truncation is nonzero, and let $N_n=0$ if the $n$-th truncation of $p$ is a zero polynomial. Then the number of distinct roots of all nonzero truncations of $p$ is
$$
N\le\sum_{n=1}^dN_n\le\sum_{n=1}^d(n-1)=\frac{(d-1)d}2
$$
Then for the set $S=\{\frac{w}{|w|}\nu^j|\;j=1,\dots,d\text{ and }w\text{ is a root of a nonzero truncation of }p\}$, we know $|S|=Nd\le\frac{(d-1)d^2}2$. If the elements of $S$ are ordered by their angle around the unit circle, then the average angle between adjacent elements is $\frac{2\pi}{|S|}\ge\frac{4\pi}{(d-1)d^2}$, and so there is at least one pair of adjacent elements that is separated by at least this amount. Thus, if we let $z_0$ be the midpoint between these two maximally separated elements on the unit circle, then the angle between $z_0$ and any element of $S$ is at least $\frac{2\pi}{(d-1)d^2}$. Thus the linear distance between $z_0$ and any element of the set $\{w\nu^j|\;j=1,\dots,d\text{ and }w\text{ is a zero of a truncation of }p\}$ is at least $\sin(\frac{2\pi}{(d-1)d^2})$.
\end{proof}
With a lower bound on the distance, a lower bound on the minimum magnitude of any component of the reduced vector can be obtained.

\begin{lemma}
\label{lem:mag-bnd}
Let $r\le1$. For any polynomial $p\in\mathcal P_d$, if there exists a $z_0$ on the unit circle such that the linear distance between any element of $\{\nu^jz_0\}_{j=1}^d$ and any zeros of any nonzero truncations of $p$ is at least $r$, then for all $j$ from $1$ to $d$
$$
|p(\nu^jz_0)|\ge\frac{r^\frac{(d-1)d}2\left(\frac{d-1}{2d}\right)^d\frac2{d-1}}{\left(\prod_{k=0}^{d-1}(r^k+1)\right)} \|\hat p\|_1 \, ,
$$
where $\|\hat p\|_1 = \sum_{j=0}^{d-1} |c_j|$, $p(z) = \sum_{j=0}^{d-1} c_j z^j$.
\end{lemma}
\begin{proof}
Let $n_0$ be the smallest $n$ obtained by applying Lemma \ref{lem:geom-trick} to $\hat p$ and $t=\frac{d-1}{2d}$. Then if we let $\hat p=(c_j)_{j=0}^{d-1}$, so that $p(z)=\sum_{j=0}^{d-1}c_jz^j$, we have
$$
|c_{n_0}|
\ge\|\hat p\|_1\left(\frac{d-1}{2d}\right)^{d-n_0}\left(\left(\frac{d-1}{2d}\right)^{-1}-1\right)
=\|\hat p\|_1\left(\frac{d-1}{2d}\right)^{d-n_0}\frac{d+1}{d-1}
$$
and for all $j<n_0$
$$
|c_j|
<\|\hat p\|_1\left(\frac{d-1}{2d}\right)^{d-j}\left(\left(\frac{d-1}{2d}\right)^{-1}-1\right)
=\|\hat p\|_1\left(\frac{d-1}{2d}\right)^{d-j}\frac{d+1}{d-1} \, .
$$
Let
$$
m(n_0,n)=\frac{r^\frac{(n-1)n}2\|\hat p\|_1\left(\frac{d-1}{2d}\right)^{d-n_0}\frac2{d-1}}{\left(\prod_{k=n_0}^{n-1}(r^k+1)\right)}
$$
We prove, by induction on $n$ from $n_0$ to $d$, that $|p_n(\nu^jz_0)|\ge m(n_0,n)$ for the $n$-th truncation $p_n\in\mathcal P_n$, and for all $j$ from $1$ to $d$.

For the base case $n=n_0$, we know that
\begin{align*}
|c_{n_0}|-\sum_{j=0}^{n_0-1}|c_j|
\ge&\|\hat p\|_1\left(\left(\frac{d-1}{2d}\right)^{d-n_0}-\sum_{j=0}^{n_0-1}\left(\frac{d-1}{2d}\right)^{d-j}\right)\frac{d+1}{d-1}\\
=&\|\hat p\|_1\left(\left(\frac{d-1}{2d}\right)^{d-n_0}-\sum_{l=d-n_0+1}^d\left(\frac{d-1}{2d}\right)^l\right)\frac{d+1}{d-1}\\
\ge&\|\hat p\|_1\left(\left(\frac{d-1}{2d}\right)^{d-n_0}-\frac{\left(\frac{d-1}{2d}\right)^{d-n_0+1}}{1-\left(\frac{d-1}{2d}\right)}\right)\frac{d+1}{d-1}\\
=&\|\hat p\|_1\left(\frac{d-1}{2d}\right)^{d-n_0}\frac2{d-1}
\end{align*}
and equality only holds if $n_0=0$. Then $|c_{n_0}|\ge\sum_{j=0}^{n_0-1}|c_j|$ and so
\begin{align*}
|p_{n_0}(\nu^jz_0)|
\ge&|c_{n_0}|-\sum_{j=0}^{n_0-1}|c_j|\\
\ge&\|\hat p\|_1\left(\frac{d-1}{2d}\right)^{d-n_0}\frac2{d-1}\\
\ge&\frac{r^\frac{(n_0-1)n_0}2\|\hat p\|_1\left(\frac{d-1}{2d}\right)^{d-n_0}\frac2{d-1}}{\left(\prod_{k=n_0}^{n_0-1}(r^k+1)\right)}\\
=&m(n_0,n_0) \, .
\end{align*}

For the inductive step, assume that we have proven that $|p_n(\nu^jz_0)|\ge m(n_0,n)$. Then we choose a threshold $\tau_n=\frac{m(n_0,n)}{r^n+1}$. If the leading coefficient $c_n$ of $p_{n+1}$ satisfies $|c_n|>\tau_n$, then $p_{n+1}$ is clearly a nonzero truncation of $p$, so by using the factored form of $p_{n+1}$, for all $j$ from $1$ to $d$,
$$
|p_{n+1}(\nu^jz_0)|\ge|c_n|r^n>\tau_nr^n=\frac{m(n_0,n)r^n}{r^n+1} \, .
$$
Otherwise, if the leading coefficient satisfies $|c_n|\le\tau_n\le m(n_0,n)$, then for all $j$ from $1$ to $d$
$$
|p_{n+1}(\nu^jz_0)|\ge m(n_0,n)-|c_n|\ge m(n_0,n)-\tau_n=m(n_0,n)-\frac{m(n_0,n)}{r^n+1}=\frac{m(n_0,n)r^n}{r^n+1} \, .
$$
Either way, for all $j$ from $1$ to $d$,
$$
|p_{n+1}(\nu^jz_0)|\ge\frac{m(n_0,n)r^n}{r^n+1}=\frac{r^\frac{(n-1)n}2\|\hat p\|_1\left(\frac{d-1}{2d}\right)^{d-n_0}\frac2{d-1}r^n}{\left(\prod_{k=n_0}^{n}(r^k+1)\right)}=m(n_0,n+1) \, .
$$
Thus, for all $j$ from $1$ to $d$,
$$
|p(\nu^jz_0)|
=|p_d(\nu^jz_0)|\ge m(n_0,d)
=\frac{r^\frac{(d-1)d}2\|\hat p\|_1\left(\frac{d-1}{2d}\right)^{d-n_0}\frac2{d-1}}{\left(\prod_{k=n_0}^{d-1}(r^k+1)\right)}
\ge\frac{r^\frac{(d-1)d}2\|\hat p\|_1\left(\frac{d-1}{2d}\right)^d\frac2{d-1}}{\left(\prod_{k=0}^{d-1}(r^k+1)\right)}
$$
\end{proof}

Using the $r$ from Lemma \ref{lem:dist-bnd} in the bound in the equation in Lemma \ref{lem:mag-bnd} gives us the desired lower bound. Note that the bound in Lemma \ref{lem:dist-bnd} was obtained by showing a worst case of equally spaced roots and the bound in Lemma \ref{lem:mag-bnd} was obtained by showing a worst case of roots that are bunched together. Thus, the lower bound on the minimum magnitude obtained by combining Lemma \ref{lem:dist-bnd} and Lemma \ref{lem:mag-bnd} will not be achieved for any polynomial of degree greater than $2$ and is thus not the greatest lower bound for higher dimensions.

\subsection{Recovery algorithms for full vectors in the presence of noise}

We also need to be able to estimate the error in the full vector recovery procedure if we know a lower bound for the magnitude measurements.

\begin{lemma}
\label{lem:full-induct-noise}
Let $m>0$. For any vectors $x\in\C^d$ and $\eps\in\R^{3d-2}$, if $\{e_j\}_{j=1}^d$ is an orthonormal basis such that for all $j$ from $1$ to $d$, $|\langle x,e_j \rangle|^2-|\eps_j|\ge m\|x\|_\infty^2$, and $C=\frac{(1+\sqrt2)\|\eps\|_\infty+\|x\|_\infty^2}{m\|x\|_\infty^2}$, then a vector $y$ may be obtained such that for all $k$ from $1$ to $d$,
$$
\left|y_k-\frac{\overline{x_1}}{|x_1|}x_k\right|\le\left(\frac{2+\sqrt2}m\frac{1-C^{k-1}}{1-C}+\frac{C^{k-1}}{2\sqrt m}\right)\frac{\|\eps\|_\infty}{\|x\|_\infty}
$$
by using the values of $\widetilde{\mathcal A}_{\{e_j\}}(x,\eps)$.
\end{lemma}
\begin{proof}
For the base case, we let $y_1=\sqrt{|\langle x,e_1 \rangle|^2+\eps_1}$. Then by the mean value theorem and concavity of the square root, there exists a $\xi$ between $|\langle x,e_1 \rangle|^2+\eps_1$ and $|\langle x,e_1 \rangle|^2$ (so that $\xi\ge|\langle x,e_1 \rangle|^2-|\eps_1|\ge m\|x\|_\infty^2>0$) such that
\begin{align*}
\left|y_1-\frac{\overline{x_1}}{|x_1|}x_1\right|
=&\left|\sqrt{|\langle x,e_1 \rangle|^2+\eps_1}-\sqrt{|\langle x,e_1 \rangle|^2}\right|\\
=&\frac{|\eps_1|}{2\sqrt\xi}\\
\le&\frac{\|\eps\|_\infty}{2\sqrt m\|x\|_\infty}\\
=&\left(\frac{2+\sqrt2}m\frac{1-C^0}{1-C}+\frac{C^0}{2\sqrt m}\right)\frac{\|\eps\|_\infty}{\|x\|_\infty}
\end{align*}

For the $k$th (with $k<d$) inductive step, we assume that we have constructed $y_k$ with the given information such that $E_k=\left|y_k-\frac{\overline{x_1}}{|x_1|}x_k\right|\le\left(\frac{2+\sqrt2}m\frac{1-C^{k-1}}{1-C}+\frac{C^{k-1}}{2\sqrt m}\right)\frac{\|\eps\|_\infty}{\|x\|_\infty}$. We abbreviate
$$
\widetilde{\overline{x_k}x_{k+1}}
=\frac12
\left( (1-i)\widetilde{\mathcal A}_{\{e_j\}}(x,\eps)_k+(1-i)\widetilde{\mathcal A}_{\{e_j\}}(x,\eps)_{k+1}
-\widetilde{\mathcal A}_{\{e_j\}}(x,\eps)_{k+d}+i \widetilde{\mathcal A}_{\{e_j\}}(x,\eps)_{k+2d+1} \right)
$$
and
$$
y_{k+1}
=\frac{\widetilde{\overline{x_k} {x}_{k+1} } }{\widetilde{\mathcal A}_{\{e_j\}}(x,\eps)_k}y_k \, .
$$
A direct computation shows the error for the approximation of the term used in phase propagation,
\begin{align*}
\left|\widetilde{\overline{x_k}x_{k+1}}-\overline{x_k}x_{k+1}\right|
=&\left|\frac{(1-i)(|x_k|^2+\eps_k)+(1-i)(|x_{k+1}|^2+\eps_{k+1})}2\right.\\
&\qquad-\frac{(|x_k-x_{k+1}|^2+\eps_{d+k})-i(|x_k-ix_{k+1}|^2+\eps_{2d+k-1})}2\\
&\qquad-\frac{(1-i)|x_k|^2+(1-i)|x_{k+1} |^2}2\\
&\qquad\left.+\frac{|x_k-x_{k+1}|^2-i|x_k-ix_{k+1}|^2}2\right|\\
=&\left|\frac{(1-i)\eps_k+(1-i)\eps_{k+1}-\eps_{d+k}+i\eps_{2d+k-1}}2\right|\\
\le&(1+\sqrt2)\|\eps\|_\infty \, .
\end{align*}
We use similar identities to simplify the relationship between the vector and approximate recovery,
\begin{align*}
\left|y_{k+1}-\frac{\overline{x_1}}{|x_1|}x_{k+1}\right|
=&\left|\frac{\widetilde{\overline{x_k}x_{k+1}}}{|\langle x,e_k \rangle|^2+\eps_k}y_k-\frac{\overline{x_k}x_{k+1}}{|x_k|^2}\frac{\overline{x_1}}{|x_1|}x_k\right|\\
=&\left|\frac{|x_k|^2\widetilde{\overline{x_k}x_{k+1}}y_k-\overline{x_k}x_{k+1}\frac{\overline{x_1}}{|x_1|}x_k(|x_k|^2+\eps_k)}{(|x_k|^2+\eps_k)|x_k|^2}\right|\\
=&\left|\frac{|x_k|^2(\widetilde{\overline{x_k}x_{k+1}}-\overline{x_k}x_{k+1})y_k+|x_k|^2\overline{x_k}x_{k+1}(y_k-\frac{\overline{x_1}}{|x_1|}x_k)-\overline{x_k}x_{k+1}\frac{\overline{x_1}}{|x_1|}x_k\eps_k}{(|x_k|^2+\eps_k)|x_k|^2}\right|\\
=&\left|\frac{(\widetilde{\overline{x_k}x_{k+1}}-\overline{x_k}x_{k+1})y_k+\overline{x_k}x_{k+1}(y_k-\frac{\overline{x_1}}{|x_1|}x_k)-x_{k+1}\frac{\overline{x_1}}{|x_1|}\eps_k}{|x_k|^2+\eps_k}\right|
\end{align*}
Next, we estimate using the triangle inequality
\begin{align*}
\left|y_{k+1}-\frac{\overline{x_1}}{|x_1|}x_{k+1}\right|
\le&\frac{\left|(\widetilde{\overline{x_k}x_{k+1}}-\overline{x_k}x_{k+1})y_k+\overline{x_k}x_{k+1}(y_k-\frac{\overline{x_1}}{|x_1|}x_k)-x_{k+1}\frac{\overline{x_1}}{|x_1|}\eps_k\right|}{m\|x\|_\infty^2}\\
\le&\frac{\left|\widetilde{\overline{x_k}x_{k+1}}-\overline{x_k}x_{k+1}\right||y_k|+|\overline{x_k}x_{k+1}||y_k-\frac{\overline{x_1}}{|x_1|}x_k|+|x_{k+1}\frac{\overline{x_1}}{|x_1|}\eps_k|}{m\|x\|_\infty^2}\\
\le&\frac{\left|\widetilde{\overline{x_k}x_{k+1}}-\overline{x_k}x_{k+1}\right|(|y_k-\frac{\overline{x_1}}{|x_1|}x_k|+|x_k|)+|\overline{x_k}x_{k+1}||y_k-\frac{\overline{x_1}}{|x_1|}x_k|+|x_{k+1}\eps_k|}{m\|x\|_\infty^2} \, .
\end{align*}
Finally, recalling that $E_k=|y_k-\frac{\overline{x_1}}{|x_1|}x_k|$ was bounded by the induction assumption,
\begin{align*}
\left|y_{k+1}-\frac{\overline{x_1}}{|x_1|}x_{k+1}\right|
\le&\frac{(1+\sqrt2)\|\eps\|_\infty(E_k+\|x\|_\infty)+\|x\|_\infty^2E_k+\|x\|_\infty\|\eps\|_\infty}{m\|x\|_\infty^2}\\
=&\frac{2+\sqrt2}{m\|x\|_\infty^2}\|x\|_\infty\|\eps\|_\infty
+\frac{(1+\sqrt2)\|\eps\|_\infty+\|x\|_\infty^2}{m\|x\|_\infty^2}E_k\\
=&\frac{2+\sqrt2}{m\|x\|_\infty}\|\eps\|_\infty+CE_k\\
\le&\frac{2+\sqrt2}{m\|x\|_\infty}\|\eps\|_\infty+C\left(\frac{2+\sqrt2}m\frac{1-C^{k-1}}{1-C}+\frac{C^{k-1}}{2\sqrt m}\right)\frac{\|\eps\|_\infty}{\|x\|_\infty}\\
=&\left(\frac{2+\sqrt2}m+\frac{2+\sqrt2}m\frac{C-C^k}{1-C}+\frac{C^k}{2\sqrt m}\right)\frac{\|\eps\|_\infty}{\|x\|_\infty}\\
=&\left(\frac{2+\sqrt2}m\frac{1-C^k}{1-C}+\frac{C^k}{2\sqrt m}\right)\frac{\|\eps\|_\infty}{\|x\|_\infty} \, .
\end{align*}
\end{proof}

\begin{rem}
\label{rem:incr}
Note that in Lemma \ref{lem:full-induct-noise}, $C\ge\frac{\|x\|_\infty^2}m\ge1$, and that the error bound
$$
\left(\frac{2+\sqrt2}m\frac{1-C^{k-1}}{1-C}+\frac{C^{k-1}}{2\sqrt m}\right)\frac{\|\eps\|_\infty}{\|x\|_\infty}
=\left(\frac{2+\sqrt2}m\left(\sum_{j=0}^{k-2}C^j\right)+\frac{C^{k-1}}{2\sqrt m}\right)\frac{\|\eps\|_\infty}{\|x\|_\infty}
$$
is increasing in $C$.
\end{rem}

\begin{thm}
Let $m>0$. For any vectors $x\in\C^d\backslash\{0\}$ and $\eps\in\R^{3d-2}$, if $\{e_j\}_{j=1}^d$ is an orthonormal basis such that for all $j$ from $1$ to $d$, $|\langle x,e_j \rangle|^2-|\eps_j|\ge m\|x\|_\infty^2$, and $\tilde x$ is the vector recovered from Lemma \ref{lem:full-induct-noise}, then an operator $\widetilde T_x:\C^d\to\C^{d-1}$ with null space equal to the set of complex multiples of $\tilde x$ can be constructed using the values of $\widetilde{\mathcal A}_{\{e_j\}}(x,\eps)$ and the basis $\{e_j\}_{j=1}^d$.
\end{thm}
\begin{proof}
As in Lemma \ref{lem:full-induct-noise}, we let
$$
\widetilde{\overline{x_k}x_{k+1}}
=\frac12
\left( (1-i)\widetilde{\mathcal A}_{\{e_j\}}(x,\eps)_k+(1-i)\widetilde{\mathcal A}_{\{e_j\}}(x,\eps)_{k+1}
-\widetilde{\mathcal A}_{\{e_j\}}(x,\eps)_{k+d}+i \widetilde{\mathcal A}_{\{e_j\}}(x,\eps)_{k+2d+1} \right) \, .
$$
With respect to the basis $\{e_j\}_{j=1}^{d}$, we define the multiplication operator $\widetilde M_{\overline{x} Sx}:\C^{d}\to\C^{d}$ by the map 
$\widetilde M_{\overline{x_j}x_{j+1}}\left((y_j)_{j=1}^{d}\right)=(\widetilde{\overline{x_k}x_{k+1}}y_j)_{j=1}^{d}$,
where as in the noiseless case, we set $x_{d+1}=0$.
Similarly, we define the multiplication operator $\widetilde M_{|x|^2}:\C^{d}\to\C^{d}$ by the map $\widetilde M_{|x|^2}\left((y_j)_{j=1}^{d}\right)=((|x_j|^2+\eps_j)y_j)_{j=1}^{d}$. Note that $\widetilde M_{|x_j|^2}$ is invertible if and only if $|\langle x,e_j \rangle|^2+\eps_j\ne0$ for all $j$ from $1$ to $d$. This is true because
$$
|\langle x,e_j \rangle|^2+\eps_j
\ge|\langle x,e_j \rangle|^2-|\eps_j|
\ge m\|x\|_\infty^2>0 \, .
$$
Let 
$S:\C^d\to\C^{d}$ be the left shift operator $S\left((y_j)_{j=1}^d\right)=(y_{j+1})_{j=1}^{d}, y_{d+1}=0$ as before. With these operators, we define the operator 
$\widetilde T_x:\C^d\to\C^{d}$ as $\widetilde T_x=\widetilde M_{|x|^2}S-\widetilde M_{\overline{x}S x}$. Then for any $c\in\C$,
\begin{align*}
\widetilde T_x(c\tilde x)
=&\widetilde M_{|x|^2}S(c\tilde x)-\widetilde M_{\overline{x}Sx}(c\tilde x)\\
=&\widetilde M_{|x|^2}\left((c\tilde x_{j+1})_{j=1}^{d}\right)-\widetilde M_{\overline{x}Sx}\left((c\tilde x_j)_{j=1}^{d}\right)\\
=&\widetilde M_{|x|^2}\left(\left(c\frac{\widetilde{\overline{x_j}x_{j+1}}}{|\langle x,e_j \rangle|^2+\eps_j}\tilde x_j\right)_{j=1}^{d}\right)
-\widetilde M_{\overline{x}Sx}\left((c\tilde x_j)_{j=1}^{d}\right)\\
=&\left(c\widetilde{\overline{x_j}x_{j+1}}\tilde x_j\right)_{j=1}^{d}-\left(\widetilde{\overline{x_j}x_{j+1}}c\tilde x_j\right)_{j=1}^{d}\\
=&0
\end{align*}
so any complex multiple of $\tilde x$ is in the null space of this operator.

Conversely, assume that $y$ is in the null space of $\widetilde T_x$. We will use an inductive argument to show that for any $j$ from $1$ to $d$, $y_j=\tilde x_j\frac{y_1}{\tilde x_1}$. Note that these quotients are  well defined, because
$$
\tilde x_1=\sqrt{|\langle x,e_1 \rangle|^2+\eps_1}\ge\sqrt{|\langle x,e_1 \rangle|^2-|\eps_1|}\ge\sqrt m\|x\|_\infty>0 \, .
$$
Then the base case $j=1$ is trivial. For the inductive step, note that for any $j$ from $1$ to $d-1$,
$$
y_{j+1}-\frac{\widetilde{\overline{x_j}x_{j+1}}}{|\langle x,e_j \rangle|^2+\eps_j}y_j
=(\widetilde M_{|x|^2}^{-1}\left(\widetilde M_{|x|^2}Sy-\widetilde M_{\overline{x}Sx}y\right))_j
=(\widetilde M_{|x|^2}^{-1}\widetilde T_xy)_j=0 \, .
$$
Thus, for any $j$ from $1$ to $d-1$,
$$
y_{j+1}
=\frac{\widetilde{\overline{x_j}x_{j+1}}}{|\langle x,e_j \rangle|^2+\eps_j}y_j
=\frac{\widetilde{\overline{x_j}x_{j+1}}}{|\langle x,e_j \rangle|^2+\eps_j}\tilde x_j\frac{y_1}{\tilde x_1}
=\tilde x_{j+1}\frac{y_1}{\tilde x_1}
$$
which shows that $y$ is a complex multiple of $\tilde x$.
\end{proof}

\begin{cor}
If $x$, $\eps$, $\{e_j\}_{j=0}^d$, $m$, and $\widetilde T_x$ are as in the preceding theorem, $\tilde x$ is the vector recovered from Lemma \ref{lem:full-induct-noise}, and $\hat x$ is the vector that solves
\begin{align*}
\hat x & = \arg\min \{ \|\widetilde T_xy\|^2: \|y\|_2^2=\sum_{j=1}^d\left(|\langle x,e_j \rangle|^2+\eps_j\right) \}
\end{align*}
then $\rho([\hat x],[x])\le2\left\|\tilde x-\frac{\overline{x_1}}{|x_1|}x\right\|_2+\frac{\sqrt d\|\eps\|_\infty}{2\sqrt m\|x\|_\infty}$
\end{cor}
\begin{proof}
By the mean value theorem and the concavity of the square root, there exists a $\xi$ between $\sum_{j=1}^d\left(|\langle x,e_j \rangle|^2+\eps_j\right)$ and $\sum_{j=1}^d|\langle x,e_j \rangle|^2$ (so that $\xi\ge\sum_{j=1}^d\left(|\langle x,e_j \rangle|^2-|\eps_j|\right)\ge dm\|x\|_\infty^2>0$) such that
$$
\sqrt{\sum_{j=1}^d\left(|\langle x,e_j \rangle|^2+\eps_j\right)}-\sqrt{\sum_{j=1}^d|\langle x,e_j \rangle|^2}
=\frac{\sum_{j=1}^d\eps_j}{2\sqrt\xi}
$$
Because the null space of $\widetilde T_x$ contains only multiples of $\tilde x$, we know $\hat x=\frac{\sqrt{\sum_{j=1}^d\left(|\langle x,e_j \rangle|^2+\eps_j\right)}}{\|\tilde x\|_2}\tilde x$ and thus, using concavity gives the estimate
\begin{align*}
\left\|\hat x-\frac{\overline{x_1}}{|x_1|}x\right\|_2
&\le\|\hat x-\tilde x\|_2+\left\|\tilde x-\frac{\overline{x_1}}{|x_1|}x\right\|_2\\
&=\left\|\frac{\sqrt{\sum_{j=1}^d\left(|\langle x,e_j \rangle|^2+\eps_j\right)}}{\|\tilde x\|_2}\tilde x-\tilde x\right\|_2+\left\|\tilde x-\frac{\overline{x_1}}{|x_1|}x\right\|_2\\
&=\left|\frac{\sqrt{\|x\|_2^2+\sum_{j=1}^d\eps_j}}{\|\tilde x\|_2}-1\right|\|\tilde x\|_2+\left\|\tilde x-\frac{\overline{x_1}}{|x_1|}x\right\|_2\\
&=\left|\sqrt{\|x\|_2^2+\sum_{j=1}^d\eps_j}-\|\tilde x\|_2\right|+\left\|\tilde x-\frac{\overline{x_1}}{|x_1|}x\right\|_2\\
&=\left|\sqrt{\|x\|_2^2}+\frac{\sum_{j=1}^d\eps_j}{2\sqrt\xi}-\|\tilde x\|_2\right|+\left\|\tilde x-\frac{\overline{x_1}}{|x_1|}x\right\|_2\\
&\le\left|\|x\|_2-\|\tilde x\|_2\right|+\frac{\sum_{j=1}^d|\eps_j|}{2\sqrt\xi}+\left\|\tilde x-\frac{\overline{x_1}}{|x_1|}x\right\|_2\\
&\le2\left\|\tilde x-\frac{\overline{x_1}}{|x_1|}x\right\|_2+\frac{\sqrt d\|\eps\|_\infty}{2\sqrt m\|x\|_\infty} \, .
\end{align*}
\end{proof}

\subsection{Main Results}
We begin the main result by showing that for a fixed $p$ with sample values that are bounded away from zero, the recovery error is $O(\|\epsilon\|_\infty)$. 
\begin{thm}
Let $\tilde m>0$. For any nonzero polynomial $p\in\mathcal P_d$, and any $\eps\in\R^{6d-3}$, if there exists a $z_0$ on the unit circle such that $\min\{|p(\nu^jz_0)|^2-(2d-1)\|\eps\|_\infty\}_{j=1}^d\ge\tilde m\|p\|_2^2$, then an approximation $\tilde p\in \mathcal P_d$ can be constructed using the Dirichlet Kernel and the values of $\widetilde{\mathcal A}(p,\eps)$, such that if $\tilde C=\frac{(1+\sqrt2)\sqrt d(2d-1)\|\eps\|_\infty+d\|p\|_2^2}{\sqrt d\tilde m\|p\|_2^2}$ then for some $c_0$ on the unit circle
$$
\|\tilde p-c_0p\|_2\le
\left(\frac{2+\sqrt2}{\tilde m}\frac{d-d\tilde C-1+\tilde C^d}{1-\tilde C}\sqrt d+\frac{1-\tilde C^d}{2\sqrt{\frac{\tilde m}{\sqrt d}}}\right)\frac{d(2d-1)}{(1-\tilde C)}\frac{\|\eps\|_\infty}{\|p\|_2} \, .
$$
\end{thm}
\begin{proof}
Let $D_{d-1}\in\mathcal T_d$ be the normalized Dirichlet kernel of degree $d-1$, so that for any $z$ in the unit circle $D_{d-1}(z)=\frac1{2d-1}\sum_{k=-(d-1)}^{d-1}z^k$. Then the set of functions $\{z\mapsto D_{d-1}(z\omega^{-l})\}_{l=1}^{2d-1}$ is orthonormal with respect to the $L^2$ inner product on the unit circle, and any $g\in\mathcal T_d$ can be interpolated as $g(z)=\sum_{l=1}^{2d-1}g(\omega^l)D_{d-1}(z\omega^{-l})$. If an error $\gamma\in\C^{2d-1}$ is present on each of the values $g(\omega^l)$, and if we let $\tilde g(z)=\sum_{l=1}^{2d-1}(g(\omega^l)+\gamma_l)D_{d-1}(z\omega^{-l})$, then for any $z$ on the unit circle
\begin{align*}
|\tilde g(z)-g(z)|
=&\left|\sum_{l=1}^{2d-1}(g(\omega^l)+\gamma_l)D_{d-1}(z\omega^{-l})-\sum_{l=1}^{2d-1}g(\omega^l)D_{d-1}(z\omega^{-l})\right|\\
=&\left|\sum_{l=1}^{2d-1}\gamma_lD_{d-1}(z\omega^{-l})\right|\\
\le&\sum_{l=1}^{2d-1}\left|\gamma_lD_{d-1}(z\omega^{-l})\right|\\
\le&(2d-1)\|\gamma\|_\infty
\end{align*}

If $h\in\mathcal P_d$, then $\overline hh\in\mathcal T_d$. Thus, each of the functions $|p(z)|^2$, $|p(z)-p(z\nu)|^2$, and $|p(z)-ip(z\nu)|^2$, are in $\mathcal T_d$, and using the Dirichlet Kernel these functions may be interpolated from the values of $\widetilde{\mathcal A}(p,\eps)$. The error present in the sample values means that approximating trigonometric polynomials are obtained from this interpolation, with a uniform error that is less than $(2d-1)\|\eps\|_\infty$
for any point on the unit circle. Let $f_0(z)=|p(z)|^2+e_0(z)$, $f_1(z)=|p(z)-p(z\nu)|^2+e_1(z)$, and $f_2(z)=|p(z)-ip(z\nu)|^2+e_2(z)$ be these approximating trigonometric polynomials.

We find a $z_0$ that satisfies the hypotheses of the theorem by a simple maximization argument on $\min\{f_0(\nu^jz)\}_{j=1}^d$. The set $\{z\mapsto\frac1{\sqrt d}K_{z_0\nu^j}(z)\}_{j=1}^d$ is an orthonormal basis for $\mathcal P_d$, and $p$ is full in this basis. Then because the values of $\sqrt df_0(z)$, $\sqrt df_1(z)$, and $\sqrt df_2(z)$ at the points $\{z_0\nu^k\}_{k=1}^d$, as well as an error $\gamma$ that depends on $e_0$, $e_1$, and $e_2$, with $\|\gamma\|_\infty\le\sqrt d(2d-1)\|\eps\|_\infty$, correspond to the measurements $\widetilde{\mathcal A}_{\{z\mapsto\frac1{\sqrt d}K_{z_0\nu^j}(z)\}}(p,\gamma)$ and we know these values on the entire unit circle, we may apply either of the full vector reconstructions given earlier to obtain an approximation $\tilde p\in\mathcal P_d$ for $p$. If Lemma \ref{lem:full-induct-noise} is applied to these measurements, and we use the equivalence of norms,
$$
\|x\|_\infty
=\max_k\{|p(z_0\nu^k)|\}
=\max_k\{|\langle p,K_{z_0\nu^k}\rangle|\}
\le\|p\|_2\|K_{z_0\nu^k}\|_2
\le\sqrt d\|p\|_2
$$
and
$$
\|x\|_\infty
=\max_k\{|p(z_0\nu^k)|\}
=\sqrt{\max_k\{|\langle p,K_{z_0\nu^k}\rangle|^2\}}
\ge\sqrt{\frac1d\sum_k|\langle p,K_{z_0\nu^k}\rangle|^2}
\ge\frac1{\sqrt d}\|p\|_2
$$
then with
$m=\frac{\sqrt d\tilde m\|p\|_2^2}{\|x\|_\infty^2}\ge\frac{\tilde m}{\sqrt d}$, and
$$
C=\frac{(1+\sqrt2)\|\gamma\|_\infty+\|x\|_\infty^2}{m\|x\|_\infty^2}
\le\frac{(1+\sqrt2)\sqrt d(2d-1)\|\eps\|_\infty+d\|p\|_2^2}{\sqrt d\tilde m\|p\|_2^2}
=\tilde C
$$
we obtain a vector of coefficients $y\in\C^d$, such that for all $k$ from $1$ to $d$,
$$
\left|y_k-\frac{\overline{p(\tilde z_0\nu)}}{|p(\tilde z_0\nu)|}p(\tilde z_0\nu^k)\right|
\le\left(\frac{2+\sqrt2}m\frac{1-C^{k-1}}{1-C}+\frac{C^{k-1}}{2\sqrt m}\right)\frac{\|\gamma\|_\infty}{\|x\|_\infty}
$$
and by Remark~\ref{rem:incr}
\begin{align*}
\left|y_k-\frac{\overline{p(\tilde z_0\nu)}}{|p(\tilde z_0\nu)|}p(\tilde z_0\nu^k)\right|
&\le\left(\frac{2+\sqrt2}m\frac{1-C^{k-1}}{1-C}+\frac{C^{k-1}}{2\sqrt m}\right)\frac{\|\gamma\|_\infty}{\|x\|_\infty}\\
&\le\left(\frac{2+\sqrt2}{\frac{\tilde m}{\sqrt d}}\frac{1-\tilde C^{k-1}}{1-\tilde C}+\frac{\tilde C^{k-1}}{2\sqrt{\frac{\tilde m}{\sqrt d}}}\right)d(2d-1)\frac{\|\eps\|_\infty}{\|p\|_2} \, .
\end{align*}
Let $\tilde p=\sum_{k=1}^dy_k\frac1{\sqrt d}K_{\tilde z_0\nu^k}$. Then Minkowski's inequality gives terms that form geometric series,
\begin{align*}
\left\|\tilde p-\frac{\overline{p(\tilde z_0\nu)}}{|p(\tilde z_0\nu)|}p\right\|_2
=&\left\|\sum_{k=1}^dy_k\frac1{\sqrt d}K_{\tilde z_0\nu^k}-\frac{\overline{p(\tilde z_0\nu)}}{|p(\tilde z_0\nu)|}\sum_{k=1}^dp(\tilde z_0\nu^k)\frac1{\sqrt d}K_{\tilde z_0\nu^k}\right\|_2\\
=&\left\|\sum_{k=1}^d\left(y_k-\frac{\overline{p(\tilde z_0\nu)}}{|p(\tilde z_0\nu)|}p(\tilde z_0\nu^k)\right)\frac1{\sqrt d}K_{\tilde z_0\nu^k}\right\|_2\\
\le&\sum_{k=1}^d\left|y_k-\frac{\overline{p(\tilde z_0\nu)}}{|p(\tilde z_0\nu)|}p(\tilde z_0\nu^k)\right|\frac1{\sqrt d}\|K_{\tilde z_0\nu^k}\|_2\\
\le&\sum_{k=1}^d\left(\frac{2+\sqrt2}{\frac{\tilde m}{\sqrt d}}\frac{1-\tilde C^{k-1}}{1-\tilde C}+\frac{\tilde C^{k-1}}{2\sqrt{\frac{\tilde m}{\sqrt d}}}\right)d(2d-1)\frac{\|\eps\|_\infty}{\|p\|_2}\\
=&\left(\frac{2+\sqrt2}{\tilde m}\frac{d-\sum_{k=1}^d\tilde C^{k-1}}{1-\tilde C}\sqrt d+\frac{\sum_{k=1}^d\tilde C^{k-1}}{2\sqrt{\frac{\tilde m}{\sqrt d}}}\right)d(2d-1)\frac{\|\eps\|_\infty}{\|p\|_2}\\
=&\left(\frac{2+\sqrt2}{\tilde m}\frac{d-d\tilde C-1+\tilde C^d}{1-\tilde C}\sqrt d+\frac{1-\tilde C^d}{2\sqrt{\frac{\tilde m}{\sqrt d}}}\right)\frac{d(2d-1)}{(1-\tilde C)}\frac{\|\eps\|_\infty}{\|p\|_2}\, .
\end{align*} 
\end{proof}

To obtain a uniform error bound that only assumes bounds on the norms of the vector $p$ and on the magnitude of the noise $\|\epsilon\|_\infty$, we use the max-min principle from Section~\ref{sec:max-min}. This provides us with a universally valid lower bound $\tilde m$ that applies to the above theorem.

\begin{thm}
\label{thm:main-uni}
Let $r=\sin(\frac{2\pi}{(d-1)d^2})$, and $0<\alpha<1$. For 
any polynomial $p\in\mathcal P_d$ with $\|p\|_2=1$, and any $\eps\in\R^{6d-3}$, if $\beta=\frac{r^\frac{(d-1)d}2\left(\frac{d-1}{2d}\right)^d\frac2{d-1}}{\left(\prod_{k=1}^{d-1}(r^k+1)\right)}$ and $\|\eps\|_\infty\le\frac{\alpha\beta^2}{2d-1}$, then an approximation $\tilde p\in \mathcal P_d$ can be reconstructed using the Dirichlet Kernel and the values of $\widetilde{\mathcal A}(p,\eps)$, such that if $\tilde C=\frac{(1+\sqrt2)(2d-1)\|\eps\|_\infty+d}{\beta^2(1-\alpha)}$ then for some $c_0$ on the unit circle
$$
\|\tilde p-c_0p\|_2\le
\left(\frac{2+\sqrt2}{\beta^2(1-\alpha)}\frac{d-d\tilde C-1+\tilde C^d}{1-\tilde C}\sqrt d+\frac{1-\tilde C^d}{2\beta\sqrt{\frac1{\sqrt d}(1-\alpha)}}\right)\frac{d(2d-1)}{(1-\tilde C)}\|\eps\|_\infty \, .
$$
\end{thm}
\begin{proof}
By Lemma \ref{lem:dist-bnd} we know that there exists a $z_0$ on the unit circle such that the distance between any element of $\{\nu^jz_0\}_{j=1}^d$ and any roots of any nonzero truncations of $p$ is at least $r$. Then by Lemma \ref{lem:mag-bnd} we know that for all $j$ from $1$ to $d$, $|p(\nu^jz_0)|\ge\beta\| \hat p\|_1\ge\beta\|p\|_2=\beta$. Thus, there exists a $z_0$ on the unit circle such that
$$\min\left\{|p(\nu^jz_0)|^2-(2d-1)\|\eps\|_\infty\right\}_{j=1}^d\ge\beta^2-(2d-1)\|\eps\|_\infty\ge\beta^2(1-\alpha)$$
for all $j$ from $1$ to $d$ and we may use $z_0$ and $\tilde m=\beta^2(1-\alpha)$ in the preceding theorem.  When we apply the  above theorem, we get
\begin{align*}
\|\tilde p-c_0p\|_2
&\le\left(\frac{2+\sqrt2}{\tilde m}\frac{d-d\tilde C-1+\tilde C^d}{1-\tilde C}\sqrt d+\frac{1-\tilde C^d}{2\sqrt{\frac{\tilde m}{\sqrt d}}}\right)\frac{d(2d-1)}{(1-\tilde C)}\|\eps\|_\infty\\
&\le\left(\frac{2+\sqrt2}{\beta^2(1-\alpha)}\frac{d-d\tilde C-1+\tilde C^d}{1-\tilde C}\sqrt d+\frac{1-\tilde C^d}{2\beta\sqrt{\frac1{\sqrt d}(1-\alpha)}}\right)\frac{d(2d-1)}{(1-\tilde C)}\|\eps\|_\infty \, .
\end{align*}
\end{proof}

We remark that any $z_0$ that satisfies the claimed max-min bound does not necessarily satisfy Lemma~\ref{lem:dist-bnd}.
This means that the above theorem would benefit immediately from an improved lower bound on the minimum magnitude.

As the final step for the main result, we remove the normalization condition on the input vector.
Since the norm of the vector enters quadratically in each component of $\mathcal A(x,\epsilon)$,
the dependence of the error bound on $\|p\|$ is not linear. Instead, we obtain a bound on the accuracy of the reconstruction
which is inverse proportional to the signal-to-noise ratio $\|p\|_2/\|\epsilon\|_\infty$, assuming
that $\|\epsilon\|_\infty$ is sufficiently small compared to $\|p\|^2$.

\begin{thm}
Let $r=\sin(\frac{2\pi}{(d-1)d^2})$, and $0<\alpha<1$. For any nonzero polynomial $p\in\mathcal P_d$, and any $\eps\in\R^{6d-3}$, if $\beta=\frac{r^\frac{(d-1)d}2\left(\frac{d-1}{2d}\right)^d\frac2{d-1}}{\left(\prod_{k=1}^{d-1}(r^k+1)\right)}$ and $\|\eps\|_\infty\le\frac{\alpha\beta^2\|p\|_2^2}{2d-1}$, then an approximation $\tilde p\in \mathcal P_d$ can be reconstructed using the Dirichlet Kernel and the values of $\widetilde{\mathcal A}(p,\eps)$, such that if $\tilde C=\frac{(1+\sqrt2)(2d-1)\|\eps\|_\infty+\sqrt d}{\beta^2(1-\alpha)}$ then for some $c_0$ on the unit circle
$$
\|\tilde p-c_0p\|_2\le
\left(\frac{2+\sqrt2}{\beta^2(1-\alpha)}\frac{d-d\tilde C-1+\tilde C^d}{(1-\tilde C)^2}+\frac{1-\tilde C^d}{2\beta\sqrt{\sqrt d(1-\alpha)}(1-\tilde C)}\right)\sqrt d(2d-1)\frac{\|\eps\|_\infty}{\|p\|_2} \, .
$$
\end{thm}
\begin{proof}
Note that $\widetilde{\mathcal A}(\frac p{\|p\|_2},\frac\eps{\|p\|_2^2})=\frac1{\|p\|_2^2}\widetilde{\mathcal A}(p,\eps)$. By Theorem \ref{thm:main-uni} we know that the values of $\widetilde{\mathcal A}(\frac p{\|p\|_2},\frac\eps{\|p\|_2^2})$ can be used to obtain
$$
\left\|\frac{\tilde p}{\|p\|_2}-c_0\frac p{\|p\|_2}\right\|_2\le
\left(\frac{2+\sqrt2}{\beta^2(1-\alpha)}\frac{d-d\tilde C-1+\tilde C^d}{1-\tilde C}\sqrt d+\frac{1-\tilde C^d}{2\beta\sqrt{\frac1{\sqrt d}(1-\alpha)}}\right)\frac{d(2d-1)}{(1-\tilde C)}\left\|\frac\eps{\|p\|_2^2}\right\|_\infty
$$
Scaling both sides of the inequality by $\|p\|$ then gives
$$
\left\|\tilde p-c_0p\right\|_2\le
\left(\frac{2+\sqrt2}{\beta^2(1-\alpha)}\frac{d-d\tilde C-1+\tilde C^d}{1-\tilde C}\sqrt d+\frac{1-\tilde C^d}{2\beta\sqrt{\frac1{\sqrt d}(1-\alpha)}}\right)\frac{d(2d-1)}{(1-\tilde C)}\frac{\|\eps\|_\infty}{\|p\|_2} \, .
$$
\end{proof}

\paragraph{Acknowledgments.} This work was supported in part by the National Science Foundation grant NSF DMS-1412524 and
by the Alexander von Humboldt foundation.

\end{document}